\documentclass[12pt]{amsart}
\usepackage{url,cite,hyperref,amsmath,amsthm,amssymb,mathtools}
\usepackage[margin=1.25in]{geometry}

\theoremstyle{plain}
\newtheorem{theorem}[subsection]{Theorem}
\newtheorem*{theorem*}{Theorem}
\newtheorem{corollary}[subsection]{Corollary}
\newtheorem{question}[subsection]{Question}
\newtheorem*{question*}{Question}

\newtheorem{lemma}[subsection]{Lemma}
\newtheorem*{lemma*}{Lemma}
\newtheorem{proposition}[subsection]{Proposition}
\newtheorem*{proposition*}{Proposition}
\theoremstyle{definition}
\newtheorem{definition}[subsection]{Definition}
\newtheorem*{definition*}{Definition}
\newtheorem{remark}[subsection]{Remark}

\newtheorem*{example*}{Example}
\newtheorem*{notation*}{Notation}

\newtheoremstyle{myclaim}
  {1ex}
  {1ex}
  {\it}
  {\parindent}
  {\it}
  {.}
  { }
  {}
\theoremstyle{myclaim}

\newtheorem*{claim*}{Claim}

\newtheoremstyle{note}
  {}
  {}
  {}
  {}
  {\normalfont}
  {.}
  {.5em}
  {}

\newtheoremstyle{citing}
  {}
  {}
  {\itshape}
  {}
  {\bfseries}
  {.}
  {.5em}
  {\thmnote{#3}}

\theoremstyle{citing}
\newtheorem*{citethm}{}

\newcommand{\Z}{\ensuremath{\mathbb{Z}}}

\newcommand{\C}{\ensuremath{\mathbb{C}}}
\newcommand{\T}{\ensuremath{\mathbb{T}}}

\newcommand{\norm}[1]{\ensuremath{\| #1 \|}}

\renewcommand{\epsilon}{\varepsilon}

\makeatletter
\newcommand*{\centerfloat}{%
  \parindent \z@
  \leftskip \z@ \@plus 1fil \@minus \textwidth
  \rightskip\leftskip
  \parfillskip \z@skip}
\makeatother
  
\let\originalleft\left
\let\originalright\right
\renewcommand{\left}{\mathopen{}\mathclose\bgroup\originalleft}
\renewcommand{\right}{\aftergroup\egroup\originalright}

\DeclareMathOperator{\Tr}{Tr}
\DeclareMathOperator{\GL}{GL}
\DeclareMathOperator{\Ind}{Ind}
\DeclareMathOperator{\rank}{rank}
\DeclareMathOperator{\qd}{qd}

\newcommand{\ZH}{\mathbb{Z}\left[\frac{1}{p}\right]}

\newcommand{\V}{\Vert}

\newcommand{\la}{\langle}
\newcommand{\ra}{\rangle}
\newcommand{\XQD}{\ensuremath{X_{\text{QD}}}}

\begin{document}

\title{%
  On groups with quasidiagonal C*-algebras
}

\author{Jos\'e R. Carri\'on}
\email{jcarrion@math.purdue.edu}
\author{Marius Dadarlat}
\email{mdd@math.purdue.edu}
\author{Caleb Eckhardt}
\email{eckharc@muohio.edu}
\thanks{C.E.\ was partially supported by NSF grant DMS-1101144}
\thanks{M.D.\ was partially supported by NSF grant DMS-1101305}
\address{%
  Department of Mathematics,
  Purdue University,
  West Lafayette, IN, 47907,
  United States}
\address{%
  Department of Mathematics,
  Miami University,
  Oxford, OH, 45056,
  United States}

\date{April 5, 2012.}

\begin{abstract}
  We examine the question of quasidiagonality for C*-algebras of
  discrete amenable groups from a variety of angles.  We give a
  quantitative version of Rosenberg's theorem via paradoxical
  decompositions and a characterization of quasidiagonality for group
  C*-algebras in terms of embeddability of the groups.  We consider
  several notable examples of groups, such as topological full groups
  associated with Cantor minimal systems and Abels' celebrated example
  of a finitely presented solvable group that is not residually
  finite, and show that they have quasidiagonal C*-algebras.  Finally,
  we study strong quasidiagonality for group C*-algebras, exhibiting
  classes of amenable groups with and without strongly quasidiagonal
  C*-algebras.
\end{abstract}

\maketitle

\section{Introduction}

In \cite{Lance73} Lance provided a C*-algebraic characterization of
amenability for discrete groups by proving that a discrete group
$\Gamma$ is amenable if and only if its reduced C*-algebra,
$C^*_r(\Gamma)$ is nuclear. Later Rosenberg showed \cite{Hadwin87}
that if $C^*_r(\Gamma)$ is quasidiagonal (see Definition
\ref{def:QD}), then $\Gamma$ is amenable, a result which has
absolutely no analog for general C*-algebras (see \cite{Dadarlat00}).
The converse to Rosenberg's theorem remains open, namely: if $\Gamma$
is a discrete, amenable group, is $C^*_r(\Gamma)$ quasidiagonal
\cite{Voiculescu93}?

The question of quasidiagonality for amenable groups is tantalizing
for a number of reasons.  First, quasidiagonality displays certain
``topological'' properties, such as homotopy invariance
\cite{Voiculescu91}.  On the other hand, one might describe
amenability as a ``measure theoretic'' property, as one can detect
amenability of $\Gamma$ in the von Neumann algebras it generates.
Hence an affirmative answer would provide a nice topological
characterization of amenability to complement its measure theoretic
description.  Second, this question is a critical test case for a
number of other open questions.  Indeed, it is not known if every
separable, nuclear and stably finite C*-algebra is quasidiagonal (a
question with important implications for the classification program)
and, much more generally, if every stably finite C*-algebra is an MF
algebra. Thus an answer to the above question concerning groups will
either provide a chain of counterexamples or some evidence to the
validity of the more general conjectures.

There are some known converses to Rosenberg's theorem.  Recall that a
group $\Gamma$ is \emph{maximally almost periodic} (MAP) if it embeds
into a compact group.  Because the C*-algebra of an amenable MAP group
is residually finite dimensional \cite{Bekka90}, it follows that the
C*-algebra of an amenable group that is the union of residually finite
groups must be quasidiagonal.  We generalize this result in
Section~\ref{subsec:group-theor-descr}.

Our main results are the following.  First, if $\Gamma$ is not
amenable, then the modulus of quasidiagonality of $C_r^*(\Gamma)$ is
controlled by the number of pieces in a paradoxical decomposition of
$\Gamma$ (Theorem~\ref{thm:qrosenberg}).  Second, if $\Gamma$ is
amenable, then $C^*(\Gamma)$ is quasidiagonal if and only if $\Gamma$
embeds in the unitary group of $\prod_{n=1}^\infty M_n(\C) \big/
\sum_{n=1}^\infty M_n(\C)$ (Theorem~\ref{thm:MF=qd2}).  We expand this
class of groups beyond the class of LEF groups of
\cite{Vershik-Gordon97}.  Third, if $\Gamma$ and $\Lambda$ are
amenable groups such that $\Gamma$ is non-torsion and $\Lambda$ has a
finite dimensional representation other than the trivial one, then
$C^*(\Lambda \wr \Gamma)$ has a non-finite quotient and therefore
cannot be strongly quasidiagonal (Theorem~\ref{thm:nonSQD}).

\subsection{Organization of the paper}

In Section~\ref{subsec:quant-vers-rosenberg} we revisit Rosenberg's
previously mentioned result.  His result implies that the modulus of
quasidiagonality \cite{Pimsner-Popa-etal79} does not vanish for some
finite subset of a non-amenable group.  The modulus of
quasidiagonality measures how badly a C*-algebra violates
quasidiagonality.  We estimate this number and a closely related one
using paradoxical decompositions, and give some calculations for free
groups.

In Section~\ref{subsec:group-theor-descr} we consider an approximate
version of MAP for groups that characterizes quasidiagonality for
discrete amenable groups.  We call groups with this property MF due to
their connection with Blackadar and Kirchberg's MF algebras
\cite{Blackadar-Kirchberg97}.  We then show that the groups that are
locally embeddable into finite groups in the sense of Vershik and
Gordon \cite{Vershik-Gordon97} (so-called LEF groups) are MF groups.
Kerr had already proved that the C*-algebra of an amenable LEF group
is quasidiagonal \cite{Kerr11}.

In Section~\ref{subsec:mf-but-not-lef}, we use our characterization of
quasidiagonality for amenable groups to give examples of solvable
groups that are not LEF but have quasidiagonal C*-algebras.  These
groups are well known examples due to Abels of finitely presented
solvable groups that are not residually finite.

Finally, in section \ref{sec:strong-quas-groups} we discuss groups and
strong quasidiagonality (see Definition~\ref{def:QD}).
Theorem~\ref{thm:nonSQD} provides examples of group C*-algebras that
are not strongly quasidiagonal, such as the C*-algebra of the
lamplighter group.  Section~\ref{subsec:strongly-quas-groups} exhibits
some classes of nilpotent groups that have strongly quasidiagonal
C*-algebras.

\subsection{Some consequences}
\label{subsec:some-consequences}

Let $X$ be the Cantor set and $T$ a minimal homeomorphism of $X$.  The
\emph{topological full group} $[[T]]$ is the group of all
homeomorphisms of $X$ that are locally equal to an integer power of
$T$.  These groups are of interest for several reasons.  For example,
they are complete invariants for flip conjugacy
\cite{Giordano-Putnam-etal99} and studying their properties as
abstract groups led to the first examples of infinite, simple,
amenable groups that are finitely generated \cite{Matui06,
  Grigorchuk-Medynets, Juschenko-Monod12}.  It follows from the
results of Section~\ref{subsec:group-theor-descr} that the C*-algebra
of $[[T]]$ must be quasidiagonal, since $[[T]]$ is LEF by
\cite{Grigorchuk-Medynets} and amenable by \cite{Juschenko-Monod12}.
Since the previously mentioned examples of infinite, simple, amenable
and finitely generated groups arise as commutator subgroups of
topological full groups associated to certain Cantor minimal systems,
their C*-algebras are quasidiagonal as well.

On the other hand, an example of Abels provides an amenable group that
is not LEF.  We observe that if a group is not LEF, then it cannot be
a union of residually finite groups and one cannot obtain
quasidiagonality based on the result of Bekka mentioned above.
However, we will see in Section~\ref{subsec:mf-but-not-lef} that the
$C^*$-algebra of Abels' example has a quasidiagonal C*-algebra.

\subsection{Definitions and notation}
\label{subsec:preliminaries}

For completeness we record the definition of quasidiagonality.  We
refer the reader to the survey article \cite{Brown04} for more
information on quasidiagonality.

\begin{definition}
  \label{def:QD}
  Let $H$ be a separable Hilbert space.  A (separable) set
  $\Omega\subset B(H)$ is \emph{quasidiagonal} if there is an
  increasing sequence of (self-adjoint) projections $(P_n) \subset
  \mathcal{K}(H)$ with $P_n \to 1_H$ strongly and such that $\norm{
    [P_n, T] } \to 0$ for every $T\in \Omega$.  (We write $[S,T]$ for
  the commutator $ST - TS$.)

  A separable $C^*$-algebra $A$ is \emph{quasidiagonal} if it has a
  faithful representation as a set of quasidiagonal operators. We say
  $A$ is \emph{strongly quasidiagonal} if $\sigma(A)$ is a
  quasidiagonal set of operators for every representation $\sigma$ of
  $A$.
\end{definition}

\begin{theorem}[Voiculescu \cite{Voiculescu91}]
  \label{thm:voiculescu-charac-qd}
  A separable $C^*$-algebra is quasidiagonal if and only if there
  exists a sequence of contractive completely positive maps
  $\phi_n\colon A\to M_{k_n}(\C)$ such that $\norm{ \phi_n(a) }\to
  \norm{a}$ and $\norm{ \phi(ab) - \phi(a)\phi(b) } \to 0$ for every
  $a,b \in A$.
\end{theorem}

In this paper we only consider discrete countable groups.  The left
regular representation of a group $\Gamma$ on $B(\ell^2\Gamma)$ maps
$s\in \Gamma$ to the operator $\lambda_s\in B(\ell^2 \Gamma)$ which is
left-translation by $s$.  For $t\in \Gamma$ we write $\delta_t \in
\ell^2\Gamma$ for the characteristic function of the set $\{ t \}$, so
that $\lambda_s \delta_t = \delta_{st}$.  The reduced C*-algebra of
$\Gamma$ is the sub-C*-algebra $C^*_r(\Gamma)$ of $B(\ell^2\Gamma)$
generated by $\lambda(\Gamma)$.  We will usually use $e$ for the
neutral element of a group $\Gamma$ and $Z(\Gamma)$ for its center. We also
write $Z(A)$ for the center of a C*-algebra $A$.

\section{Quasidiagonality and groups}
\label{sec:qd-groups}

\subsection{A quantitative version of Rosenberg's theorem}
\label{subsec:quant-vers-rosenberg}

In \cite{Hadwin87} Rosenberg proved that if a group $\Gamma$ is not
amenable, then $C_r^*(\Gamma)$ is not quasidiagonal.  First we
reformulate his result.

\begin{definition}
  \label{def:CFG}
  Let $\mathcal{P}$ be the set of non-zero finite-rank projections on
  $\ell^2\Gamma$.  Given a finite subset $F\subset \Gamma$, set
  \begin{equation*}
    \label{eq:constantdefn}
    C_{F} :=
    \inf_{P\in \mathcal{P}}
    \sup_{x\in F}
    \|[\lambda_x,P] \|.
  \end{equation*}
\end{definition}

It is clear that if $C^*_r(\Gamma)$ is quasidiagonal, then
$C_{F}=0$ for every finite subset of $\Gamma$. Furthermore, if
$\Gamma$ is amenable, then $\lambda$ has an approximately fixed
vector, so $C_{F}=0$ for all finite subsets as well.  Rosenberg
\cite{Hadwin87} has proved that if $\Gamma$ is not amenable, then
there is a finite subset $F\subseteq \Gamma$ such that
$C_{F}>0$.  In this section we give a quantitative version of
this statement by estimating (and in some cases calculating)
$C_{F}$ using paradoxical decompositions of $\Gamma$.

We point out a very similar concept due to Pimsner, Popa and
Voiculescu \cite{Pimsner-Popa-etal79}.  Recall that the \emph{modulus
  of quasidiagonality} of a set $\Omega\subset B(\ell^2\Gamma)$ is
\begin{displaymath}
  \qd(\Omega) :=
  \liminf_{P\in \mathcal{P}}
  \sup_{T\in \Omega}
  \|[T,P] \|,
\end{displaymath}
where the order on projections is given by $P\leq Q$ if $PQ = P$.
Clearly $C_F \leq \qd(\lambda(F))$.

Recall that a group $\Gamma$ is not amenable if and only if it admits
a \emph{paradoxical decomposition}: that is, there exist pairwise
disjoint subsets $X_1,\dots,X_n,Y_1,\dots,Y_m\subseteq \Gamma $ and
$g_1, \dots, g_n, h_1, \dots, h_m \in \Gamma$ with $g_1 = h_ 1 = e$ such
that
\begin{equation}
  \label{eq:pardec}
  \Gamma=\bigsqcup_{i=1}^n g_i X_i
  = \bigsqcup_{j=1}^m h_j Y_j
  = \bigg( \bigsqcup_{i=1}^n X_i \bigg)
    \sqcup \bigg( \bigsqcup_{j=1}^m Y_j \bigg).
\end{equation}
In this case we say that the paradoxical decomposition has $n+m$
\emph{pieces}.  A non-amenable group always has a paradoxical
decomposition with at least four pieces.  It is well known that a
group contains a copy of the free group on two generators if and only
if one can find a paradoxical decomposition with exactly 4 pieces.
(We refer the reader to \cite{Wagon85} for more information on
paradoxical decompositions.)

We will require an elementary lemma.
\begin{lemma}
  \label{lem:Tracelemma} Let $H$ be a Hilbert space and let $\Tr$
  denote the usual trace on $B(H)$. Let $X=X^*\in B(H)$ be finite rank
  with $\Tr(X)=0$. Then for any $Q\in B(H)$ with $0\leq Q\leq 1$ we
  have
  \begin{equation*}
    |\Tr(QX)|\leq \frac{1}{2}\textup{rank}(X)\V X\V.
  \end{equation*}
\end{lemma}

\begin{proof}
  If $Y$ is a finite rank operator, then $\Tr(Y)\leq \rank(Y)\|Y\|$.
  Indeed, if $E$ is a projection onto the range of $Y$, then $\Tr(Y) =
  \Tr(EY)\leq \Tr(E)\|Y\| = \rank(Y)\|Y\|$.

  Write $X = X_{+} - X_{-}$ with $X_\pm \geq 0$ and $X_{+}X_{-} = 0$.
  Then $\Tr(X_{\pm})\leq \rank(X_\pm)\|X_\pm\|\leq \rank(X_\pm) \|X\|$.
  Since $\rank(X) = \rank(X_{+}) + \rank(X_{-})$ and $\Tr(X_{+}) =
  \Tr(X_{-})$ we obtain
  \[
  \Tr(X_{\pm}) \leq \frac{1}{2}\rank(X)\|X\|.
  \]
  Now, since
  \[
  \Tr(QX) = \Tr( Q^{1/2}X_{+}Q^{1/2} ) - \Tr( Q^{1/2}X_{-}Q^{1/2} )
  \]
  and $\Tr( Q^{1/2}X_{\pm}Q^{1/2} )\leq \|Q\|\Tr( X_{\pm} )\leq
  \Tr( X_{\pm} ),$ it follows that
  \[
  -\frac{1}{2}\rank(X)\|X\|
  \leq - \Tr(X_{-}) \leq \Tr(QX)\leq \Tr(X_{+})
  \leq \frac{1}{2}\rank(X)\|X\|.
  \]
\end{proof}

\begin{theorem}
  \label{thm:qrosenberg}
  Suppose $\Gamma$ is a non-amenable group with a paradoxical
  decomposition as in (\ref{eq:pardec}).  If $F=\{ g_1,\dots,g_n,
  h_1,\dots,h_m \}$, then
  \begin{equation*}
    C_{F}\geq \frac{1}{n+m-2}.
  \end{equation*}
  In particular, if $\Gamma$ contains $\mathbb{F}_2$, then $C_{F}\geq
  1/2$ by choosing a minimal decomposition with four pieces.

  Since $\qd(\lambda(F)) \geq C_F$ we have the same statement for the
  modulus of quasidiagonality of $\lambda(F)$ instead of $C_F$.
\end{theorem}

\begin{proof}
  For each subset $A\subseteq \Gamma$, let $P_A$ be the projection
  onto $\overline{\textrm{span}}\{\delta_a:a\in A \}$. Let $\Tr$ denote
  the usual semi-finite trace on $B(\ell^2\Gamma)$ and let $P\in
  B(\ell^2\Gamma)$ be a finite-rank projection of rank $k\geq 1$. Suppose
  that $\V [\lambda_x,P] \V\leq \epsilon$ for all $x\in F$. We prove
  $\epsilon\geq \frac{1}{n+m-2}$.

  Let $1\leq i\leq n$. By Lemma \ref{lem:Tracelemma},
  \begin{equation}
    \label{eq:tracestimate}
    \big| \Tr\bigl( P_{g_iX_i}(P-\lambda_{g_i}P\lambda_{g_i^{-1}}) \big) \bigr|
    \leq k\epsilon.
  \end{equation}
  Because $\lambda_{g_i} P_{X_i} \lambda_{g_i^{-1}} = P_{g_i X_i}$, we have
  \[
  \Tr(P_{g_iX_i}P)
  =
  \Tr( P_{X_i}P )
    + \Tr\big( P_{g_iX_i}(P-\lambda_{g_i}P\lambda_{g_i^{-1}}) \big).
  \]
  From this and the estimate (\ref{eq:tracestimate}) it follows that
  for each $2\leq i\leq n$
  \begin{equation}
    \label{eq:giXiineq}
    \Tr(P_{g_iX_i}P) \leq \Tr(P_{X_i}P) +k\epsilon.
  \end{equation}
  Let $X=\cup X_i$ and $Y=\cup Y_j$. By (\ref{eq:giXiineq}) and the
  fact that $g_1=e$ we obtain
  \begin{align}
    \label{eq:Xtrace}
    k = \Tr(P)
      & =\Tr(P_{X_1}P)
          + \sum_{i=2}^n \Tr(P_{g_iX_i}P)\\ \notag
      & \leq \Tr(P_{X_1}P)+\sum_{i=2}^n \Tr(P_{X_i}P)
          + (n-1)k\epsilon\\ \notag
      & = \Tr(P_XP)+(n-1)k\epsilon.
  \end{align}
  From a similar calculation involving the $h_i$'s and $Y_i$'s we see
  that
  \begin{equation}
    \label{eq:Ytrace}
    k\leq \Tr(P_YP)+(m-1)k\epsilon.
  \end{equation}
  Finally, add up (\ref{eq:Xtrace}) and (\ref{eq:Ytrace}) to obtain
  the conclusion.
\end{proof}

Now we calculate $C_{F}$ when $\Gamma=\mathbb{F}_2$. Let us fix some
notation first.  Let $a,b$ be generators of $\mathbb{F}_2$ and for
each word $w\in\mathbb{F}_2$ let $|w|$ denote the word length of $w$
with respect to the generating set $\{ a, b, a^{-1}, b^{-1} \}$.  For
each $n\geq0$ let $S_n$ denote the sphere of radius $n$, that is,
\begin{equation*}
  S_n=\{ w\in \mathbb{F}_2: |w|=n \}.
\end{equation*}
Note that $ S_0=\{e\}$. For each $x\in \{ a, b, a^{-1}, b^{-1} \}$,
let $S_n^x$ denote those elements of $S_n$ whose first letter is $x$.
It is easy to see that
\begin{equation}
  \label{eq:comb}
  |S_n| = 4\cdot 3^{n-1}
  \quad \textrm{ and } \quad
  |S_n^x| = 3^{n-1}
  \quad \textrm{ for }n \geq 1.
\end{equation}
It is well known that (see \cite[Theorem 4.2]{Wagon85}) there is a
paradoxical decomposition of $\mathbb{F}_2$ as
\begin{equation}
  \label{eq:paradec}
  \mathbb{F}_2
  = X_1\sqcup aX_2
  = Y_1\sqcup b Y_2
  = X_1\sqcup X_2\sqcup Y_1\sqcup Y_2 .
\end{equation}

\begin{theorem}
  For any $\epsilon>0$ and any $n\geq 1$ there is a projection $P\in
  B(\ell^2\mathbb{F}_2)$ of rank $n$ such that $\V [\lambda_a,P] \V<
  1/2+\epsilon$ and $\V [\lambda_b,P] \V<1/2+\epsilon$. In particular,
  \begin{equation*}
    C_{\{ a,b\}} = 1/2.
  \end{equation*}
\end{theorem}

\begin{proof}
  By Voiculescu's Weyl-von Neumann type theorem, $\lambda\colon
  \mathbb{F}_2\to B(\ell^2\mathbb{F}_2)$ is approximately unitarily
  equivalent to $\lambda\otimes 1_n\colon \mathbb{F}_2\to
  B(\ell^2\mathbb{F}_2\otimes \C^n)$.  On the other hand
  $\|[\lambda_x\otimes 1_n,P\otimes 1_n]\|=\|[\lambda_x,P]\|$ for
  $P\in B(\ell^2\mathbb{F}_2)$.  It follows that it suffices to prove
  the first part of the statement for $n=1$.

  Let $P\in B(\ell^2\mathbb{F}_2)$ be any projection and $U\in
  B(\ell^2\mathbb{F}_2)$ a unitary.  Since in a C*-algebra
  $\|x^*x\|=\|x\|^2$, using the identity $[U,P]=(1-P)UP-PU(1-P)$ we
  see that $\V [U,P] \V=\textrm{max}\{ \V PU(1-P) \V, \V (1-P)UP \V
  \}$. Moreover, if $P$ is rank 1 and $\xi$ is a norm one vector in
  its range, then
  \begin{align*}
  \V (1-P)UP \V^2 & =\V (1-P)UP\xi \V^2=\V (1-P)U\xi \V^2\\
  &=\|U\xi\|^2-\|PU\xi\|^2=1-|\la U(\xi),\xi \ra|^2.
  \end{align*}
  From the above observations it now suffices to find, for each
  $\epsilon>0$, a norm 1 vector $\xi\in \ell^2\mathbb{F}_2$ such that
  \begin{equation}
    \label{eq:biggiship}
    |\la \lambda_x(\xi),\xi \ra|
    > \frac{\sqrt{3}}{2}-\epsilon
    \quad \textrm{ for }\quad
    x\in \{ a,a^{-1},b,b^{-1}\}.
  \end{equation}
  Let $n>\sqrt{3}/2\epsilon$.
  Define 
  $\alpha_i=(|S_i|n)^{-1/2}$
  and
  \begin{equation*}
    \xi = \sum_{i=1}^n \alpha_i \bigg(  \sum_{x\in S_i} \delta_x \bigg).
  \end{equation*}
  It is clear that $\V \xi \V=1$. We then have
  \begin{align*}
    \la \lambda_a(\xi),\xi \ra
      & = \sum_{i=1}^{n}\sum_{j=1}^{n} \bigg\la \sum_{x\in S_i}\alpha_i\delta_{ax},\sum_{y\in S_j}\alpha_j\delta_y \bigg\ra\\
      & = \sum_{i=1}^{n}\sum_{j=1}^{n} \bigg\la  \sum_{x\in (S_{i-1}\setminus S_{i-1}^a)\cup S_{i+1}^a}\alpha_i\delta_x, \sum_{y\in S_j}\alpha_j\delta_y \bigg\ra \\
      & = \sum_{i=1}^{n} \bigg\la  \sum_{x\in (S_{i-1}\setminus S_{i-1}^a)\cup S_{i+1}^a}\alpha_i\delta_x, \sum_{y\in S_{i-1}}\alpha_{i-1}\delta_y+\sum_{z\in S_{i+1}}\alpha_{i+1}\delta_z \bigg\ra \\
      & = \sum_{i=1}^{n}\big( \alpha_i\alpha_{i-1}|S_{i-1}\setminus S_{i-1}^a|+\alpha_i\alpha_{i+1}|S_{i+1}^a|\, \big)\\
      & \geq \sum_{i=2}^{n}\bigg(\frac{1}{4n\sqrt{3^{i-1}}\sqrt{3^{i-2}}}(3)(3^{i-2})
      +\frac{1}{4n\sqrt{3^i}\sqrt{3^{i-1}}}3^i\bigg)\quad \textrm{(by (\ref{eq:comb}))}\\
      & = \frac{\sqrt{3}}{2}-\frac{\sqrt{3}}{2n}>\frac{\sqrt{3}}{2}-\epsilon.
  \end{align*}
  The corresponding inequality for $\lambda_b$ follows by symmetry.
  Since $|\la \lambda_x(\xi),\xi \ra| = |\la \lambda_{x^{-1}}(\xi),\xi
  \ra|$, this proves (\ref{eq:biggiship}). We complete the proof by
  applying Theorem~\ref{thm:qrosenberg} to the paradoxical
  decomposition of $\mathbb{F}_2$ given in (\ref{eq:paradec}).
\end{proof}

\subsection{A characterization of quasidiagonality for amenable
  groups}
\label{subsec:group-theor-descr}

For each increasing sequence $\vec{n}=(n_k)$ of positive integers, we
consider the C*-algebra
\[
Q_{\vec{n}}= \prod_k M_{n_k}(\C) \Big/ \sum_k M_{n_k}(\C).
\]
Here the C*-algebra $\prod_k M_{n_k}(\C)$ consists of all sequences
$(a_k)$ of matrices $a_k\in M_{n_k}(\C)$ such that
$\sup_k\|a_k\|<\infty$ and $\sum_k M_{n_k}(\C)$ is the two-sided
closed ideal consisting of those sequences with the property that
$\lim_{k\to \infty}\|a_k\|=0$.

Recall that a separable C*-algebra is MF if it embeds as a
sub-C*-algebra of $Q_{\vec{n}}$ for some $\vec{n}$, see
\cite{Blackadar-Kirchberg97}.  In analogy with the class of MF
algebras, we make the following definition.

\begin{definition}
  \label{def:mf}
  A countable group $\Gamma$ is \emph{MF} if it embeds in the unitary
  group of $Q_{\vec{n}}$ for some $\vec{n}$.
\end{definition}

It is readily seen that $\Gamma$ is MF if and only if it embeds in
$U(Q_{\vec{n}})$ where $\vec{n} = (1, 2, 3, \dots)$.

Recall that a group $\Gamma$ is called \emph{maximally almost
  periodic} (abbreviated MAP) if it embeds in a compact
group. Equivalently, $\Gamma$ embeds in
\[
U\bigg( \prod^\infty_{k=1}M_k(\C)\bigg)
= \prod^\infty_{k=1} U(k).
\]
A discrete residually finite group is MAP.

Bekka \cite{Bekka-Louvet00} proved that if $\Gamma$ is a discrete
countable amenable group, then $\Gamma$ is maximally almost periodic
if and only if $C^*(\Gamma)$ is residually finite dimensional.  That
is,
\begin{displaymath}
  \Gamma \hookrightarrow U\bigg( \prod_{k=1}^\infty M_k(\C) \bigg)
  \quad\Leftrightarrow\quad
  C^*(\Gamma) \hookrightarrow \prod_{k=1}^\infty M_k(\C).
\end{displaymath}
The following theorem says that a discrete countable amenable group
$\Gamma$ embeds in $U(Q_{\vec{n}})$ for some sequence $\vec{n}$ if and
only if $C^*(\Gamma)$ embeds in $Q_{\vec{m}}$ for some $\vec{m}$.

\begin{theorem}
  \label{thm:MF=qd2}
  Let $\Gamma$ be a discrete countable amenable group. Then $\Gamma$
  is MF if and only if the C*-algebra $C^*(\Gamma)$ is quasidiagonal.
\end{theorem}

For the proof we will need the following result from
\cite{Dadarlat-Deaconu11}.

\begin{proposition}[Prop.~2.1 of \cite{Dadarlat-Deaconu11}]
  \label{prop:residually-AF}
  Let $\Gamma$ be a discrete amenable group.  Suppose there exist a
  sequence $(B_k)_{k=1}^\infty$ of unital C*-algebras and a sequence
  $(\omega_k)_{k=1}^\infty$ of group homomorphisms $\omega_k\colon
  \Gamma\to U(B_k)$ that separate the points of $\Gamma$, and that
  each $\omega_k$ appears infinitely many times in the sequence. Then
  $C^*(\Gamma)$ embeds unitally into the product C*-algebra $
  \prod_{n=1}^\infty C_n $, where $C_n = \bigotimes_{k=1}^n
  M_2(\C)\otimes B_k\otimes B_k$ (minimal tensor product).
\end{proposition}

\begin{proof}[Proof of Theorem~\ref{thm:MF=qd2}]
  By the Choi-Effros lifting theorem \cite{Choi-Effros76} and the
  local characterization of quasidiagonality given by Voiculescu,
  Theorem~\ref{thm:voiculescu-charac-qd}, it follows that a separable
  and nuclear C*-algebra $A$ is quasidiagonal if and only if $A$ is an
  MF algebra.  Suppose that $\Gamma$ is MF.  Then there is an
  injective homomorphism $\theta:\Gamma\to U(Q_{\vec{n}})$ for some
  $\vec{n}$. Let $B$ be the sub-C*-algebra of $Q_{\vec{n}}$ generated
  by $\theta(\Gamma)$.  Since $\Gamma$ is amenable, $B$ is nuclear and
  hence quasidiagonal.  By Proposition~\ref{prop:residually-AF}
  $C^*(\Gamma)$ embeds the product $\prod_n C_n$, where $C_n =
  M_{2^n}(\C)\otimes B^{\otimes 2n}$. Since $B$ is quasidiagonal, so
  is each $C_n$. It follows that $C^*(\Gamma)$ is quasidiagonal.

  Conversely, if $C^*(\Gamma)$ is quasidiagonal, then
  $C^*(\Gamma)\subset Q_{\vec{n}}$ for some $\vec{n}$ and hence
  $\Gamma$ is MF.
\end{proof}

\begin{definition}
  A group $\Gamma$ is \emph{locally embeddable into the class of finite
    groups} (or simply LEF) if for any finite subset ${F}\subset
  \Gamma$ there is a finite group $H$ and a map $\phi\colon \Gamma\to H$
  that is both injective and multiplicative when restricted to $F$.
\end{definition}

\begin{remark}
  \label{rem:rfd-vs-lef}
  Vershik and Gordon introduced LEF groups in \cite{Vershik-Gordon97}.
  Theorems~1 and~2 of \cite{Vershik-Gordon97} illustrate the
  relationship between LEF groups and residually finite groups.
  Specifically, if every finitely generated subgroup of a group
  $\Gamma$ is residually finite, then $\Gamma$ is LEF.  On the other
  hand, every finitely presented LEF group is residually finite.
  There are finitely presented solvable non-LEF groups, see
  \cite{Abels79}.
\end{remark}

We will show that an LEF group is MF.  The following lemma will
be used.

\begin{lemma}
  \label{lem:lef-iff-embeds}
  Let $\Gamma$ be a discrete countable group. Then $\Gamma$ is LEF if
  and only if there is a sequence of finite groups $( H_k
  )_{k=1}^\infty$ and an injective homomorphism $\Phi\colon \Gamma\to
  \prod_k H_k / \sum_k H_k$.
\end{lemma}

\begin{proof}
  Suppose first that $\Gamma$ embeds in a group of the form $\prod_k
  H_k / \sum_k H_k$ where $H_k$ are finite groups.  Fix a finite
  subset $F$ of $\Gamma$. Let $\phi:\Gamma \to \prod_k H_k$ be a set
  theoretic lifting of $\Phi$. It is necessarily an injective
  map. Write $\phi(s)=(\phi_k(s))$ for maps $\phi_k:\Gamma \to
  H_k$. Since any sequence in $\sum_k H_k$ has at most finitely many
  nontrivial terms it follows that there is $k_0$ such that for all
  $k\geq k_0$ the map $\phi_k:\Gamma\to H_k$ is multiplicative and
  injective on $F$.

  Conversely, suppose that $\Gamma$ is LEF.  Write $\Gamma=\bigcup_k F_k$
  where $(F_k)_k$ is an increasing sequence of finite subsets of
  $\Gamma$. Since $\Gamma$ is LEF, there is a sequence of finite
  groups $(H_k)_k$ and maps $\phi_k:\Gamma\to H_k$ such that each
  $\phi_k$ is multiplicative for on $F_k$ and the restriction of
  $\phi_k$ to $F_k$ is injective. It is then immediate that the
  sequence $(\phi_k)$ induces an embedding of groups $\Phi:\Gamma\to
  \prod_k H_k / \sum_k H_k$.
\end{proof}


\begin{proposition}
  \label{prop:lef-implies-mf}
  Let $\Gamma$ be a countable discrete group. If $\Gamma$ is LEF, then
  $\Gamma$ is MF.
\end{proposition}

\begin{proof}
  Let $H_k$ and $\phi_k$ be as in the proof of
  Lemma~\ref{lem:lef-iff-embeds}.  Let $\lambda_k\colon H_k\to
  B(\ell^2(H_k))$ be the left regular representation of $H_k$.  If
  $s,t\in F_k$, $s\neq t$, then $\phi_k(s)\neq \phi_k(t)$ and
  hence $\|\lambda_k(\phi_k(s))-\lambda_k(\phi_k(t))\|\geq
  \sqrt{2}$.  Let $n_k=|H_k|$ and set $\vec{n}=(n_k)$ as above.
  Consider the maps $\pi_k=\lambda_k\circ \phi_k:\Gamma\to U(n_k)$.
  Then the sequence of maps $(\pi_k)$ induces a group homomorphism
  $\pi\colon \Gamma \to U(Q_{\vec{n}})$ which is injective since
  $\|\pi(s)-\pi(t)\|\geq \sqrt{2}$ whenever $s,t$ are distinct
  elements of $\Gamma$.
\end{proof}

\subsection{An MF group that is not LEF}
\label{subsec:mf-but-not-lef}

Let $p$ be a prime number.  We recall the following group from
\cite{Abels79}:
\begin{equation*}
  \Gamma =
  \left\{%
    \begin{pmatrix}
      1 & x_{12} & x_{13} & x_{14}\\
      0 & p^k & x_{23} & x_{24}\\
      0 & 0 & p^n & x_{34}\\
      0 & 0 & 0 & 1
    \end{pmatrix}
    : x_{ij}\in \ZH,\ k,n\in \Z
  \right\}.
\end{equation*}
One sees that
\begin{equation*}
  Z(\Gamma) =
  \left\{%
    \begin{pmatrix}
      1 & 0 & 0 & x\\
      0 & 1 & 0 & 0\\
      0 & 0 & 1 & 0\\
      0 & 0 & 0 & 1
    \end{pmatrix}
    :  x\in \ZH
  \right\}.
\end{equation*}
We define
\begin{equation*}
  N =
  \left\{
    \begin{pmatrix}
      1 & 0 & 0 & k\\
      0 & 1 & 0 & 0\\
      0 & 0 & 1 & 0\\
      0 & 0 & 0 & 1
    \end{pmatrix}
    : k\in \Z
  \right\}.
\end{equation*}
Abels showed in \cite{Abels79} that $\Gamma$ and $\Gamma/N$ are
finitely presented groups. Moreover, as observed by Abels, ideas
similar to those in \cite[Page 349]{Hall61}, show that $\Gamma/N$ does
not have the Hopf property.  It is well known that any finitely
generated residually finite group has the Hopf property (see
\cite{Malcev40, Magnus69}); hence $\Gamma/N$ is not residually
finite. In particular $\Gamma/N$ is not LEF (see
Remark~\ref{rem:rfd-vs-lef}).  As observed in \cite{Abels79}, the
group $\Gamma/N$ is a solvable---hence amenable---group.  Finally, we
note that $\Gamma$ is residually finite.  This follows from Mal'cev's
theorem (stating that a finitely generated subgroup of $\GL(n, F)$ is
residually finite for any field $F$) \cite{Malcev40, Magnus69}.

The proof of the following lemma basically consists of writing down
the definitions of induced representations, which we do for the
convenience of the reader.

\begin{lemma}
  \label{lem:induced}
  Let $F$ be a finite group and $H\leq Z(F)$. Let $\gamma: H \to L(E)$
  be a finite dimensional unitary representation of $H$, and
  $\Ind_H^F\gamma$ be the induced unitary representation of $F$.  Then
  \begin{enumerate}
  \item[(i)] $\V \Ind_H^F\gamma(g)- 1\V\geq\sqrt{2}$ for all $g\in
    F\setminus H$ and
  \item[(ii)] $\Ind_H^F\gamma |_H$ is unitary equivalent to $[F:H]\, \gamma$.
 \end{enumerate}
\end{lemma}

\begin{proof}
  Recall that (see e.g. \cite[Appendix E]{Bekka-Harpe-etal08})
  $\textrm{Ind}_H^F\gamma$ is defined on the Hilbert space
  \begin{equation*}
    \mathcal{A} =
    \{ \xi\colon F\to E: \xi(xh) = \gamma(h^{-1})\xi(x)
    \quad \textrm{ for all }x\in F, h\in H \},
  \end{equation*}
  with inner product defined by
  \begin{equation*}
    \la \xi,\eta \ra =
    \sum_{xH\in F/H} \la\xi(x),\eta(x)\ra.
  \end{equation*}
  (Note that if $xH=yH$, then $\la\xi(x),\eta(x)\ra
  =\la\xi(y),\eta(y)\ra$ so the above inner product is well defined).
  One then defines the induced representation on the finite
  dimensional Hilbert space $\mathcal{A}$ by the equations
  \begin{equation*}
    \Ind_H^F\chi(g)\xi(x) = \xi(g^{-1}x)
    \quad \textrm{ for all }g,x\in F.
  \end{equation*}
  Suppose now that $g\not\in H$. Define $\eta\in \mathcal{A}$ by
  \begin{equation*}
    \eta(x) =
    \begin{cases}
      \gamma(x^{-1})\xi_0 & \text{ if }x\in H\\
      0            & \text{ if }x\not\in H,
    \end{cases}
  \end{equation*}
  where $\xi_0\in E$ is a unit vector.  Then $\V\eta \V=1$ and
  \begin{equation*}
    \la \Ind_H^F \gamma(g) \eta,\eta \ra
    = \la \eta(g^{-1}), \eta(e) \ra
    = 0.
  \end{equation*}
  This proves (i).

  Let $h\in H$. Then
  \begin{equation*}
   ( \Ind_H^F \gamma(h) \xi)(x)
    = \xi(h^{-1}x)
    = \xi(xh^{-1})
    = \gamma(h) \xi(x).
  \end{equation*}
  This proves (ii).
\end{proof}

\begin{remark}
  \label{rem:wyswyg}
  For a group $G$ we denote by $\lambda_G$ its left regular
  representation.  Let $\Gamma$ be a discrete countable residually
  finite group. It follows that there is a decreasing sequence of
  finite index normal subgroups $(L_n)_{n\geq 1}$ of $\Gamma$ such
  that $\bigcap_{n=1}^\infty L_n=\{e\}$.  We denote by $\pi_n$ the
  corresponding surjective homomorphisms $\pi_n:\Gamma\to
  \Gamma_n:=\Gamma/L_n$.  It is known (and not hard to verify) that
  $\lambda_\Gamma$ is weakly contained in the direct sum of the
  representations $ \lambda_{\Gamma_n}\circ \pi_n$.  If in addition
  $\Gamma$ is amenable, then it follows that the set of irreducible
  subrepresentations of all of $\lambda_{\Gamma_n}\circ \pi_n$ is
  dense in the primitive spectrum of $\Gamma$.
\end{remark}

\begin{theorem}
  \label{thm:on-center-MF}
  Let $\Gamma$ be a countable discrete residually finite group and let
  $N$ be a central subgroup of $\Gamma$. Then $\Gamma/N$ is MF and
  hence if in addition $\Gamma$ is amenable, then $C^*(\Gamma/N)$ is
  quasidiagonal.
\end{theorem}

\begin{proof}
  We will construct a sequence of finite-dimensional unitary
  representations $\sigma_n$ of $\Gamma$, such that
  \begin{equation}
    \label{eq:sepfam}
    \lim_{n\rightarrow\infty} \V \sigma_n(x)-1 \V = 0
    \quad \textrm{ if and only if }\quad x\in N.
  \end{equation}
  In particular, this will prove that $\Gamma/N$ is MF.  Indeed,
  writing $\sigma_n\colon \Gamma\to U(k(n))$, one sees that the map of
  $G/N$ to $U( \prod_n M_{k(n)} / \sum_{n} M_{k(n)} )$ given by
  $x\mapsto (\sigma_n(x))$ is an embedding.

  Let $L_n$, and $\pi_n:\Gamma\to \Gamma_n:=\Gamma/L_n$ be as in
  Remark~\ref{rem:wyswyg}. Let $Z$ be the center of $\Gamma$ and set
  $Z_n=\pi_n(Z)\cong Z/Z\cap L_n$. The restriction of $\pi_n$ to $Z$
  is denoted again by $\pi_n:Z\to Z_n$.  Let $\hat{\pi}_n:\hat{Z}_n\to
  \hat{Z}$ be the dual map of the restriction of $\pi_n$ to $Z$. It
  follows that the union of $\hat{\pi}_n(\hat{Z}_n)$ is dense in
  $\hat{Z}$ as noted in Remark~\ref{rem:wyswyg} applied to $Z$ and its
  finite index subgroups $Z\cap L_n$.  Let $(\omega_i)_{i\geq 1}$ be a
  dense sequence in the Pontriagin dual $(Z/N)\,\hat{}$. We will
  regard $\omega_i$ as characters on $Z$ that are trivial on $N$.  Set
  $\eta_n=\omega_1\oplus\cdots\oplus \omega_n:Z\to U(n)$.  Let us
  observe that
  \begin{equation}
    \label{eqn:10012}
    \lim_{n\to \infty}\|\eta_n(x)-1\|=0,
    \quad \text{if and only if}\quad x\in N.
  \end{equation}
  Write $Z$ as an increasing union of finite subsets $F_n$.  Since the
  union of $\hat{\pi}_n(\hat{Z}_n)$ is dense in $\hat{Z}$, we can
  replace $\Gamma_n$ by $\Gamma_1\oplus \Gamma_2\oplus \cdots \oplus
  \Gamma_{m(n)}$, $\pi_n$ by $\pi_1\oplus \pi_2\oplus \cdots \oplus
  \pi_{m(n)}$ and $Z_n$ by $Z_1\oplus Z_2\oplus \cdots \oplus
  Z_{m(n)}$ so that in the new setup $\pi_n(Z)\subset Z_n$, $Z_n$ is a
  central subgroup of $\Gamma_n$ and the following properties will
  hold.

  \begin{itemize}
  \item[(i)] For each $n\geq 1$ there is a unitary representation
    $\gamma_n: Z_n \to U(n)$ such that
    \begin{equation}
      \label{eqn:10013}
      \|\eta_n(x)-\gamma_n\circ\pi_n(x)\|<1/n, \quad\text{for all}\, x\in F_n.
    \end{equation}
  \item[(ii)] For any $x\in \Gamma\setminus Z$ there is $m$ such that
    $\pi_n(x)\notin Z_n$ for all $n\geq m$.
  \end{itemize}
  Concerning (ii) let us note that if $x\in \Gamma$ and $\pi_n(x)\in
  Z_n$ for all $n\geq m$ then $\pi_n(xgx^{-1}g^{-1})=1$ for all $g\in
  \Gamma$ and $n\geq m$ and that implies that $x\in Z$ since the
  sequence $(\pi_n)_{n\geq m}$ separates the points of $\Gamma$.

  Define the finite dimensional unitary representation of $\Gamma$
  \begin{equation*}
    \sigma_n = ( \Ind^{\Gamma_n}_{Z_n}\gamma_n )\circ \pi_{n}.
  \end{equation*}
  Let $x\in  Z$.  By Lemma \ref{lem:induced}(2) we have
  \begin{equation*}
    \| \sigma_n(x)-1\| =\|\gamma_n\circ \pi_n(x)-1\|
  \end{equation*}
  and hence in conjunction with \eqref{eqn:10013}
  \begin{equation*}
    \lim_{n\to \infty} \| \sigma_n(x)-1\|=0 \, \Leftrightarrow  \,
    \lim_{n\to \infty} \|\gamma_n\circ \pi_n(x)-1\|=0
    \, \Leftrightarrow  \, \lim_{n\to \infty} \|\eta_n(x)-1\|=0.
  \end{equation*}
  By \eqref{eqn:10013}, it follows that $\V \sigma_n(x)-1
  \V\rightarrow 0$ if and only if $x\in N$.

  Now let $x\in \Gamma\setminus Z(\Gamma)$. By (ii) there is an $m$
  large enough so $\pi_{n}(x)\not\in Z_n$ for all $n\geq m$. By
  Lemma~\ref{lem:induced}(1), we have $\V \sigma_n(x)-1
  \V\geq\sqrt{2}$ for all $n\geq m$.

  This proves that \eqref{eq:sepfam} holds and therefore that
  $\Gamma/N$ is MF.
\end{proof}

\begin{corollary}\label{cor:on-center-MF}
  Let $\Gamma$ and $N$ be as in \ref{subsec:mf-but-not-lef}.  Then
  $\Gamma/N$ is MF but not LEF.  Since $\Gamma/N$ is amenable it also
  follows that $C^*(\Gamma/N)$ is quasidiagonal.
\end{corollary}

\begin{proof}
  This follows from Theorem~\ref{thm:on-center-MF}.  We have already
  noted that $\Gamma/N$ cannot be LEF since it is finitely presented
  but not residually finite.
\end{proof}

\begin{remark}
  In \cite[Section~3.2]{Thom08a} Thom shows that every quotient of a
  hyperlinear group by a central subgroup is again hyperlinear.  He
  gives an example of a group $K$ with property (T) that is not
  initially subamenable but is hyperlinear.  The proof of
  Theorem~\ref{thm:on-center-MF} and his proof that $K$ is hyperlinear
  have a similar flavor (cf. \cite[Section~3.2]{Thom08a}), but we are
  unaware of any direct connections.  We also note that these examples
  are related to the example of Abels used above (see also
  \cite{Cornulier07}).
\end{remark}

\section{Strong quasidiagonality and groups}
\label{sec:strong-quas-groups}

We exhibit some classes of amenable groups that have non-strongly
quasidiagonal C*-algebras.  See Theorem~\ref{thm:nonSQD}.  All of these
groups arise as wreath products.  We do not know if these C*-algebras
are quasidiagonal, except for a certain subclass.  See Proposition
~\ref{prop:wreath-qd}.

Let us establish some notation related to crossed product C*-algebras.
(We refer the reader to \cite[Section 4.1]{Brown-Ozawa08} for
details.)  Let $A$ be a unital C*-algebra, $\Gamma$ a discrete
countable group, and $\alpha\colon\Gamma\rightarrow \text{Aut}(A)$ a
homomorphism.  A $*$-representation $(\pi,H)$ of $A$ induces
$*$-representation $\pi\times\lambda$ of $A\rtimes_\alpha \Gamma$ on
$B(H\otimes\ell^2\Gamma)$ defined by
\begin{align}
  \label{eq:cross-rep}
  (\pi\times\lambda)(a) (\xi\otimes \delta_t)
  & = \pi\big( \alpha_{t^{-1}}(a) \big)(\xi)\otimes \delta_{t}\\
    (\pi\times\lambda)(s) (\xi\otimes \delta_t)
  & = \xi\otimes \delta_{st} \notag
\end{align}
for $a\in A$, $s, t\in \Gamma$, $\xi\in H$, and where
$\{\delta_s\}_{s\in \Gamma}$ is the usual orthonormal basis of
$\ell^2\Gamma$.

We denote by $A^{\otimes\Gamma}$ as the $\Gamma$-fold maximal tensor
product of $A$ with itself.  This infinite tensor product is defined
as an inductive limit indexed by finite subsets of $\Gamma$.  The
Bernoulli action $\beta$ of $\Gamma$ on $A^{\otimes\Gamma}$ may be
described formally by
\begin{equation*}
  \beta_s(a_{t_1}\otimes \cdots \otimes a_{t_n})
  = a_{st_1}\otimes \cdots \otimes a_{st_n}.
\end{equation*}

The proof of the next proposition was inspired by Theorem
25~\cite{Hadwin87}.
\begin{proposition}
  \label{prop:Hadwingen}
  Let A be a unital C*-algebra which is generated by two proper
  two-sided closed ideals.  Let $\Gamma$ be a discrete countable
  non-torsion group.  Then $A^{\otimes \Gamma}\rtimes_\beta \Gamma $
  has a non-finite quotient. In particular, it is not strongly
  quasidiagonal.
\end{proposition}

\begin{proof}
  Write $A=I_1+I_2$ where $I_i$ are proper two-sided closed ideals of
  $A$ and write $1_A=y+x$ where $y\in I_1$ and $x\in I_2$.  Let
  $\pi_i$ be a unital representations of $A$ with kernel $I_i$,
  $i=1,2$. Then $\pi_1(x)=1$ and $\pi_2(x)=0$.

  Let $u\in \Gamma$ be an element of infinite order. It generates a
  subgroup $\mathbb{Z}\leq \Gamma$.  Choose a subset $F\subseteq
  \Gamma$ of left coset representatives, that is,
  \begin{equation*}
    \Gamma= \bigsqcup_{s\in F}g\mathbb{Z}
  \end{equation*}
  Set
  \begin{equation*}
    \Gamma_1= \{ su^n:n<0\textrm{ and }s\in F \},
    \quad \Gamma_2=\{ su^n:n\geq0\textrm{ and }s\in F \}
  \end{equation*}
  and observe that $u\Gamma_1^{-1}$ is a proper subset of
  $\Gamma_1^{-1}$.  Define the representation $(\pi,H)$ of
  $A^{\otimes\Gamma}$ by
  \begin{equation*}
    \pi:= \bigg(\bigotimes_{s\in \Gamma_1}\pi_1\bigg)
    \otimes \bigg(\bigotimes_{s\in \Gamma_2}\pi_2\bigg).
  \end{equation*}

  For $t\in \Gamma$, let $x_t\in A^{\otimes\Gamma}$ be the elementary
  tensor with $x$ in the ``$t$''-position and $1$ elsewhere, in
  particular $\beta_g(x_t)=x_{gt}$. It follows from the properties of
  $\pi_1$, $\pi_2$ and $x$ that
  \begin{equation}\label{eq:good}
    \pi(x_t)=1_H\text{ if } t\in \Gamma_1,
    \quad
    \text{and}
    \quad
    \pi(x_t)=0\text{ if } t \in \Gamma_2.
  \end{equation}
  For a set $S\subset \Gamma$ we denote by $\chi_S$ the characteristic
  function of $S$ as well as the corresponding multiplication operator
  by $\chi_S$ on $\ell^2\Gamma$.  Using \eqref{eq:cross-rep} and
  \eqref{eq:good} one verifies immediately that
  \[(\pi\times\lambda)(x_e)=1_H\otimes \chi_{\Gamma_1^{-1}}.\]
  Define the partial isometry
  \begin{equation*}
    V:= (\pi\times\lambda)(u)\cdot (\pi\times \lambda)(x_e)
      = 1_H\otimes \lambda(u)\chi_{\Gamma_1^{-1}} .
  \end{equation*}
  Then $V^*V=1_H\otimes \chi_{\Gamma_1^{-1}}$ and $VV^* =1_H\otimes
  \chi_{u\Gamma_1^{-1}}$.  It follows that $V^*V - VV^* = 1\otimes
  \chi_{\Gamma_1^{-1}\setminus u\Gamma_1^{-1}} > 0$.
\end{proof}

\begin{corollary}
  \label{cor:nontcenter}
  Let $A$ be a unital C*-algebra which admits a quotient with
  nontrivial center.  Let $\Gamma$ be a discrete countable non-torsion
  group.  Then $A^{\otimes \Gamma}\rtimes \Gamma$ has a non-finite
  quotient.
\end{corollary}

\begin{proof}
  If $B$ is a quotient of $A$, then $B^{\otimes \Gamma}\rtimes \Gamma$
  is a quotient of $A^{\otimes \Gamma}\rtimes \Gamma$. Thus we may
  assume that the center $Z(A)$ is non-trivial. Write $Z(A)$ as the
  sum of two maximal ideals $Z(A)=J_1+J_2$.  We conclude the proof by
  applying Proposition~\ref{prop:Hadwingen} for the ideals $I_1=J_1A$ and
  $I_2=J_2A$.
\end{proof}

\begin{remark}
  \label{rem:maxideal}
  The condition on $A$ in Proposition~\ref{prop:Hadwingen} is
  equivalent to the requirement that the primitive spectrum of $A$
  contains two non-empty disjoint closed subsets.  It is not hard to
  see that the primitive spectrum of a separable $A$ contains two
  distinct closed points if and only if $A$ has two distinct maximal
  ideals. Moreover, in this case $A$ admits a quotient with nontrivial
  center.
\end{remark}

Although we state the next result in greater generality, perhaps the
most interesting case is when the groups are amenable.

\begin{theorem}
  \label{thm:nonSQD}
  Let $\Gamma$ be a discrete countable non-torsion group and let
  $\Lambda$ be any discrete countable group such that either
  \begin{enumerate}
  \item $\Lambda$ admits a finite dimensional
    representation other than the trivial representation, or
  \item $\Lambda$ has a finite conjugacy class.
  \end{enumerate}
  Then, $C^*(\Lambda \wr \Gamma)$ has a non-finite quotient.
\end{theorem}

\begin{proof}
  We first notice that $C^*(\Lambda \wr \Gamma)\cong
  C^*(\Lambda)^{\otimes \Gamma}\rtimes_\beta \Gamma$. We first assume
  (1).  By assumption, there are two inequivalent finite dimensional
  irreducible representations $\pi_1$ and $\pi_2$ of $C^*(\Lambda)$.
  Setting $I_i=\ker(\pi_i)$, $i = 1,2$, we see that $I_1$ and $I_2$
  are distinct maximal ideals that satisfy the hypothesis of
  Proposition~\ref{prop:Hadwingen}.

  Now assume (2). It is well known that if $\Lambda$ has a finite
  conjugacy class, then $C^*(\Lambda)$ has a non-trivial center
  (simply add up the elements of the conjugacy class to produce a
  central element). The conclusion follows from Corollary
  \ref{cor:nontcenter}.
\end{proof}

We observe that if $\Gamma$ is as in Theorem~\ref{thm:nonSQD} and
$\Lambda$ is not amenable, then the same conclusion follows.  Indeed,
in this case the C*-algebra $C^*_r(\Lambda)$ cannot have a
character. Thus $\ker(\lambda)$ is not contained in the kernel
$I_1\lhd C^*(\Lambda)$ of the trivial representation, but in some
other maximal ideal $I_2$ of $C^*(\Lambda)$.  Hence $I_1$ and $I_2$
are distinct maximal ideals of $C^*(\Lambda)$ and we can apply
Proposition~\ref{prop:Hadwingen}.

As a special case of Theorem~\ref{thm:nonSQD} we obtain the simplest
example we know of an amenable group with a non-strongly quasidiagonal
C*-algebra.

\begin{corollary}
  \label{cor:lamplighter}
  The C*-algebra of the group $\Z/2\Z \wr \Z$ is not strongly
  quasidiagonal.
\end{corollary}

Since $\Z/2\Z \wr \Z$ is residually finite (one may find a separating
family of homomorphisms $\pi_n\colon \Z/2\Z \wr \Z \to \Z/2\Z \wr
\Z/n\Z$), its C*-algebra is quasidiagonal (it is actually residually finite dimensional and in particular inner
quasidiagonal, see \cite{Blackadar-Kirchberg01} for relevant
definitions).  More generally, we have the following:

\begin{proposition}
  \label{prop:wreath-qd}
  Let $\Lambda$ be an amenable group that is a union of residually
  finite groups. Then the C*-algebra of the group $\Lambda \wr
  \mathbb{Z}^k$ is quasidiagonal.
\end{proposition}

\begin{proof}
  Write $\Lambda$ as an increasing union of residually finite groups
  $\Lambda_i$.  Then $\Lambda \wr \Z^k$ is the union of $\Lambda_i \wr
  \Z^k$. Therefore, we may assume that $\Lambda$ itself is residually
  finite.  By \cite{Dadarlat05}, $C^*(\Lambda)$ embeds in the
  $\mathrm{UHF}$ algebra of type $2^\infty$, denoted here by $D$.
  Then
  \[
  C^*(\Lambda \wr \mathbb{Z}^k)
  \cong
  \bigg( \bigotimes_{\Z^k} C^*(\Lambda) \bigg) \rtimes \Z^k
  \subset
  \bigg( \bigotimes_{\Z^k} D \bigg) \rtimes \Z^k
  \]
  and $\big(\bigotimes_{\Z^k} D\big) \rtimes \Z^k\cong D\rtimes
  \Z^k$ embeds in a simple unital AF algebra by a result of N.~Brown
  \cite{Brown00a}.
\end{proof}

If $C^*(\Lambda)$ has two distinct maximal ideals we do not need to
assume (1) or (2) in Theorem~\ref{thm:nonSQD} to obtain its
conclusion.  This raises an interesting question.

\begin{question}
  Are there any non-trivial groups $\Lambda$ such that $C^*(\Lambda)$
  has a unique maximal ideal? (Such a group would have to be
  amenable).
\end{question}

\subsection{Groups with strongly
quasidiagonal C*-algebras}
\label{subsec:strongly-quas-groups}

Now we exhibit some classes of (amenable) groups $\Gamma$ with
strongly quasidiagonal C$^*$-algebras.  These will arise as extensions
\begin{displaymath}
  1\to \Delta\to \Gamma\to \Lambda\to 1
\end{displaymath}
where $\Delta$ is a central subgroup of $\Gamma$, with some additional
hypotheses on $\Lambda$ and $\Delta$.  For example, we have the
following proposition.

\begin{proposition}
  \label{prop:abelian-by-abelian-sQD}
  Suppose $\Gamma$ has a central subgroup $\Delta$ such that both
  $\Delta$ and $\Gamma/\Delta$ are finitely generated abelian groups.
  Then $C^*(\Gamma)$ is strongly quasidiagonal.
\end{proposition}

\begin{proof}
  Theorem~2.2 of \cite{Carrion11} shows that such groups have finite
  decomposition rank.  A C*-algebra with finite decomposition rank
  must be strongly quasidiagonal, as proved in
  \cite[Theorem~5.3]{Kirchberg-Winter04}.
\end{proof}

The use of decomposition rank only serves to simplify the exposition,
although it is perhaps of interest in itself.  Proving strong
quasidiagonality in all the cases obtained here could be done using
analogous permanence properties of strong quasidiagonality.

\begin{lemma}
  \label{lem:cont-field-QD-fibers}
  Let $A$ be a separable continuous field C*-algebra over a locally
  compact and metrizable space $X$.  Write $A_x$ for the fiber at
  $x\in X$.  If $A$ is nuclear, then the set
  \[
  \XQD := \{ x\in X : A_x\text{ is quasidiagonal}\, \}
  \]
  is closed.
\end{lemma}

\begin{proof}
  Let $y\in \overline{\XQD}$.  Fix a finite subset $\mathcal{F}$ of
  $A_y$ and $\epsilon > 0$.  For $x\in X$ write $\pi_x\colon A\to A_x$
  for the quotient map.  Because $A_y$ is nuclear, the Choi-Effros
  lifting theorem affords us a contractive completely positive lift
  $\psi\colon A_y\to A$ of $\pi_y$.  Using the fact that $x\mapsto
  \norm{ \pi_x(\tilde{a}) }$ is continuous for every $\tilde{a}\in A$
  we see that the sets
  \[
  U = \bigcap_{a\in \mathcal{F}}
  \{ x\in X : \norm{ \pi_x( \psi(a) ) } > \norm{a} - \epsilon \}
  \]
  and
  \[
  V = \bigcap_{(a,b)\in \mathcal{F}\times\mathcal{F}}
  \{ x\in X : \norm{ \pi_x( \psi(ab) - \psi(a)\psi(b) ) } < \epsilon \}
  \]
  are finite intersections of open sets containing $y\in
  \overline{\XQD}$.  Then there exists $z\in \XQD\cap U\cap
  V$.

  Let $\phi = \pi_z\circ \psi\colon A_y\to A_z$.  This is a completely
  positive contraction satisfying \( \norm{ \phi(a) } > \norm{ a } -
  \epsilon \) and $\norm{ \phi(ab) - \phi(a)\phi(b) } < \epsilon$ for
  all $a,b\in \mathcal{F}$.  It follows from
  Theorem~\ref{thm:voiculescu-charac-qd} that $A_y$ is quasidiagonal.
\end{proof}

The integer Heisenberg group $\mathbb{H}_3$ has a central subgroup
isomorphic to $\Z$ such that the corresponding quotient of
$\mathbb{H}_3$ is $\Z^2$.  Therefore $\mathbb{H}_3$ satisfies the
conditions of Proposition~\ref{prop:abelian-by-abelian-sQD}.  One
could also consider the case when $\Gamma$ has a central subgroup
$\Delta$ such that $\Gamma/\Delta \cong \mathbb{H}_3$.

\begin{proposition}
  Suppose $\Gamma$ has a finitely generated central subgroup $\Delta$
  such that $\Gamma/\Delta$ is a finitely generated, torsion-free,
  two-step nilpotent group with rank one center.  Then $C^*(\Gamma)$
  is strongly quasidiagonal.
\end{proposition}

\begin{proof}
  As noted in Section~2 of \cite{Lee-Packer95}, the analysis of
  \cite[Corollary~3.4]{Baggett-Packer94} shows that every discrete,
  finitely generated, torsion-free, two-step nilpotent group with rank
  one center is isomorphic to a ``generalized discrete Heisenberg
  group'' $H(d_1, \dots, d_n)$.  If $n$ is a positive integer and
  $d_1, \dots, d_n$ are positive integers with $d_1 | \dots | d_n$,
  then $H(d_1, \dots, d_n)$ is defined as the set $\Z \times \Z^n
  \times \Z^n$ with multiplication
  \begin{displaymath}
    (r,s,t)\cdot(r',s',t') = (r+r' + \sum d_{i}t_{i}s_{i}', s+s', t+t').
  \end{displaymath}

  Write $\Gamma/\Delta \cong H(d_1, \dots, d_n)$.

  \emph{Case $n>1$.}  It follows from Theorem~3.4 of
  \cite{Lee-Packer95} that every twisted group algebra
  $C^*(\Gamma/\Delta, \sigma)$ of $\Gamma/\Delta$ is isomorphic to the
  section algebra of a continuous field of C*-algebras over a
  one-dimensional space with each fiber stably isomorphic to a
  noncommutative torus of dimension at most $2n$.  Every
  noncommutative torus of dimension at most $2n$ has decomposition
  rank at most $4n + 1$ (by \cite[Lemma~4.4]{Carrion11}) and
  decomposition rank is invariant under stable isomorphism, so
  Lemma~4.1 of \cite{Carrion11} implies that $C^*(\Gamma/\Delta,
  \sigma)$ has finite decomposition rank.  Now, by Theorem~1.2 of
  \cite{Packer-Raeburn92} we have that $C^*(\Gamma)$ is itself a
  continuous field C*-algebra over the finite-dimensional space
  $\hat{\Delta}$ with every fiber isomorphic to some twisted group
  C*-algebra of $\Gamma/\Delta$.  By Lemma~4.1 of \cite{Carrion11} we
  obtain that $C^*(\Gamma)$ has finite decomposition rank and is
  therefore strongly quasidiagonal.

  \emph{Case $n=1$.}  Write $H$ for $H(d_1)$.  There is an isomorphism
  $H^2(H, \T) \cong \T^2$ such that whenever a multiplier $\sigma$
  corresponds to $(\lambda, \mu)\in \T^2$ with both of $\lambda$ and
  $\mu$ torsion elements, then the twisted group C*-algebra $C^*(H,
  \sigma)$ is stably isomorphic to a noncommutative torus
  \cite[Theorem~3.9]{Lee-Packer95}.  When at least one of $\lambda$ or
  $\mu$ is non-torsion, we have that $C^*(H, \sigma)$ is simple and
  has a unique trace \cite[Theorem~3.7]{Lee-Packer95}.  Now, there is
  a continuous field of C*-algebras over $H^2(H, \T)$ where the fiber
  over (the class of) a given multiplier $\sigma$ is $C^*(H, \sigma)$
  \cite[Corollary~1.3]{Packer-Raeburn92}.  Since the fibers are
  quasidiagonal over a dense set of points, every fiber must be
  quasidiagonal (by Lemma~\ref{lem:cont-field-QD-fibers}).  In fact,
  every fiber must be strongly quasidiagonal, owing either to
  simplicity or to having finite decomposition rank.

  By Theorem~1.2 of \cite{Packer-Raeburn92} we have that $C^*(\Gamma)$
  is the section algebra of a continuous field of C*-algebras with
  strongly quasidiagonal fibers (the fibers are of the form $C^*(H,
  \sigma)$).  Therefore, every primitive quotient of $C^*(\Gamma)$ is
  strongly quasidiagonal, since it is a primitive quotient of some
  fiber of the field \cite[Theorem~10.4.3]{Dixmier77}.  It follows
  from Proposition~5 of \cite{Hadwin87} that $C^*(\Gamma)$ is strongly
  quasidiagonal.
\end{proof}

\begin{citethm}[Conjecture]
  If $\Gamma$ is a finitely generated countable discrete nilpotent
  group, then $C^*(\Gamma)$ is strongly quasidiagonal.
\end{citethm}

A group is \emph{supramenable} if it contains no paradoxical subsets.
(A subset is paradoxical if it admits a paradoxical decomposition as
in \ref{subsec:quant-vers-rosenberg}.)  For solvable groups, this is
the same as saying it has a nilpotent subgroup of finite index.  Does
every countable discrete supramenable group have a (strongly)
quasidiagonal C*-algebra?

\subsection*{Added in proof}
The third author has announced \cite{Eckhardt13} a proof of the above
conjecture (in an even stronger form).

\bibliographystyle{abbrv}

\end{document}